\documentclass[12pt,a4paper]{amsart}
\usepackage{times}
\usepackage{amscd,amssymb,amsthm,latexsym,amsmath}
\usepackage{amsfonts,latexsym,amssymb,url}
\usepackage{hyperref}
\usepackage{pifont}
\usepackage{color}
\usepackage{lipsum}

\newtheorem{theorem}{Theorem}[section]
\newtheorem{cor}[theorem]{Corollary}
\newtheorem{lemma}[theorem]{Lemma}
\newtheorem{prop}[theorem]{Proposition}
\newtheorem{defi}[theorem]{Definition}
\newtheorem{remark}[theorem]{Remark}
\theoremstyle{definition}
\newtheorem{example}[theorem]{Example}
\newcommand{\R}{{\mathbb R}}
\newcommand{\N}{{\mathbb N}}

\numberwithin{equation}{section}

\title[Inclusion between Stummel classes and other function spaces]
{Inclusion between generalized Stummel classes and other function spaces}

\author[Nicky K.~Tumalun, Denny I. Hakim, and Hendra Gunawan]{Nicky K.~Tumalun$^1$, Denny I. Hakim$^2$, and Hendra Gunawan$^3$
\vspace{0.1 cm}
\\
Analysis and Geometry Group\\ Bandung Institute of Technology\\
Jl. Ganesha No. 10, Bandung 40132, Indonesia
\vspace{0.1 cm}
\\
E\lowercase{mail: $^1$nickytumalun@yahoo.co.id, $^2$dennyivanalhakim@gmail.com, $^3$hgunawan@math.itb.ac.id}
}

\begin{document}

\maketitle

\begin{abstract}
We refine the definition of generalized Stummel classes and study inclusion properties
of these classes. We also study the inclusion relation between Stummel classes and other
function spaces such as generalized Morrey spaces, weak Morrey spaces, and Lorentz spaces. In addition, we show that these inclusions are proper. Our results extend some previous results in \cite{CRR, RZ}.


\end{abstract}

\section{Introduction}

The definition of Stummel class was introduced in \cite{DH, RZ}. For $0<\alpha<n$,
the \textit{Stummel class} $S_{\alpha}=S_\alpha(\R^n)$ is defined by
\begin{align*}
S_{\alpha}:=  \left\{ f\in L_{\rm loc}^{1}(\R^n): \eta_{\alpha}f(r) \searrow 0
\quad \text{for} \quad r \searrow 0 \right\},
\end{align*}
where
\begin{align*}
\eta_{\alpha}f(r) := \sup_{x \in \R^{n}} \int_{|x-y|<r} \frac{|f(y)|}{|x-y|^{n-\alpha}} \,dy,
\quad r>0.
\end{align*}
For $\alpha=2$, $S_{2}$ is known as the {\it Stummel-Kato class}. Knowledge of Stummel classes
is important when one is studying the regularity properties of the solutions of some partial
differential equations (see \cite{AS, CRR, CFG, Fazio, M}).

In the mean time, the study of Morrey spaces, which were introduced by C.~B.~Morrey in \cite{Mor},
has attracted many researchers, especially in the last two decades. For $1 \leq p < \infty$ and
$0\le\lambda \le n$, the {\it Morrey space} $L^{p,\lambda}=L^{p,\lambda}(\R^n)$ is defined to be
the collection of all functions $f\in L_{\rm loc}^p(\R^n)$ for which
\begin{align*}
\|f\|_{L^{p,\lambda}} := \sup_{x\in\R^n,\,r>0} r^{-\frac{\lambda}{p}}\| f \|_{L^{p}(B(x,r))}
< \infty,
\end{align*}
where $ B(x,r) := \{ y \in \R^{n} : |x-y|<r \}$ and
\begin{equation*}
\| f \|_{L^{p}(B(x,r))} := \left( \int_{|x-y|<r} |f(y)|^{p} \,dy\right)^{\frac{1}{p}}.
\end{equation*}
Note that $L^{p,0}=L^p$.
As shown in \cite{RZ}, one may observe that $L^{1,\lambda} \subseteq S_{\alpha}$ provided that
$n - \lambda < \alpha < n$. (For the case $\alpha=2$, this fact was proved in \cite{Fazio}.)
Conversely, if $V \in S_{\alpha}$ for $0<\alpha<n$ and $\eta_{\alpha}f(r) \sim r^{\sigma}$ for
some $\sigma > 0$, then $V \in L^{1,n-\alpha+\sigma}$.

Eridani and Gunawan \cite{EG} developed the concept of generalized Stummel classes and studied
the inclusion relation between these classes and generalized Morrey spaces. For $1\le p<\infty$
and a measurable function $\Psi: (0,\infty) \rightarrow (0,\infty)$, the {\it generalized
Morrey space} $L^{p,\Psi}=L^{p,\Psi}(\R^n)$ is the collection of all
functions $f \in L^{p}_{\rm loc}(\R^{n})$ for which
\begin{align*}
\| f \|_{L^{p,\Psi}} := \sup_{x\in\R^n,\,r>0} \frac{|B(x,r)|^{-\frac{1}{p}}}{\Psi(r)}
\| f \|_{L^{p}(B(x,r))} < \infty,
\end{align*}
where $|B(x,r)|$ denotes the Lebesgue measure of $B(x,r)$. Observe that, for $\Psi(t) :=
t^{\frac{\lambda -n}{p}}$ ($0 \leq \lambda \leq n$),
we have $L^{p,\Psi} = L^{p,\lambda}$. Further works on the inclusion relation between
generalized Stummel classes and Morrey spaces can be found in \cite{G, S}.

The purpose of this paper is to refine the definition of generalized Stummel classes and
study the inclusion relation between these classes. We also study the inclusion relation
between Stummel classes and Morrey spaces using assumptions that are different from the
assumptions used in \cite{EG, G, S}. We give an example of a function which belongs to the
generalized Stummel class but not to the generalized Morrey space. Furthermore, we
prove that the Stummel class contains weak Morrey spaces under certain conditions.
For $1\le p<\infty$ and $0 \leq \lambda \leq n $, the {\it weak Morrey space}
$wL^{p,\lambda}=wL^{p,\lambda}(\R^n) $ is the collection of all
Lebesgue measurable functions $f$ on $ \R^{n} $ which satisfy
\begin{align*}
\| f \|_{wL^{p,\lambda}} := \sup_{x\in\R^n,\,r>0} r^{-\frac{\lambda}{p}}
\| f \|_{wL^{p}(B(x,r))} < \infty,
\end{align*}
where
$$
\| f \|_{wL^{p}(B(x,r))} := \sup\limits_{t>0} t\left| \left\lbrace y \in B(x,r):
|f(y)|>t \right\rbrace  \right|^{\frac{1}{p}}.
$$
Observe that, by taking $\lambda=0$, we can recover the weak Lebesgue space $wL^p$.
In this paper, we also study the relation between Stummel classes and Lorentz spaces.

Throughout this paper we assume that $\Psi: ( 0,\infty) \rightarrow (0,\infty) $ is a
measurable function. Whenever required, we consider the following conditions on $\Psi$:

\begin{equation}\label{eq 1.1}
\int_{0}^{1} \frac{\Psi (t)}{t} dt < \infty;
\end{equation}
\begin{equation}\label{eq 1.2}
\frac{1}{A_{1}} \leq \frac{\Psi(s)}{\Psi(r)} \leq A_{1} \quad \text{for} \quad 1 \leq
\frac{s}{r} \leq 2;
\end{equation}
\begin{equation}\label{eq 1.3}
\frac{\Psi(r)}{r^{n}} \leq A_{2} \frac{\Psi(s)}{s^{n}} \quad \text{for} \quad s \leq r,
\end{equation}
where $ A_{i} > 0$, $ i = 1, 2 $, are independent of $r,s > 0$.
The condition (\ref{eq 1.2}) is known as the \textit{doubling condition} on $\Psi$.
In some cases, we can weaken the doubling condition by the {\it right doubling condition}:
\begin{equation}\label{eq 1.4}
\frac{\Psi(s)}{\Psi(r)} \leq A_{3} \quad \text{for} \quad 1 \leq \frac{s}{r} \leq 2,
\end{equation}
where $A_{3}$ is independent of $r, s > 0$.

In this paper, the constant $c>0$ that appears in the proof of all theorems may vary
from line to line, and the notation $c = c(\alpha, \beta, \dots, \zeta)$ indicates that
$c$ depends on $\alpha, \beta, \dots, \zeta$.

\section{The Generalized Stummel classes}

\begin{defi}\label{D1}
For $1 \leq p < \infty$, we define the \textbf{generalized Stummel $p$-class}
$S_{p,\Psi}=S_{p,\Psi}(\R^n)$ by
\begin{align*}
S_{p,\Psi}:=\left\{f\in L_{\rm loc}^{p}(\R^n):
\eta_{p,\Psi}f(r) \searrow 0 \quad \text{for} \quad r \searrow 0 \right\},
\end{align*}
where
\begin{align*}
\eta_{p,\Psi}f(r) := \sup_{x \in \R^{n}} \left( \int_{|x-y|<r}
\frac{|f(y)|^{p} \Psi(|x-y|)}{|x-y|^{n}} \, dy\right)^{\frac{1}{p}},\quad r>0.
\end{align*}
\end{defi}

We call $ \eta_{p,\Psi}f $ the {\it Stummel $p$-modulus} of $f$. Observe that the
Stummel $p$-modulus is nondecreasing on $(0,\infty)$. For $p = 1$, we have
$S_{1,\Psi} := S_{\Psi}$ --- the generalized Stummel class introduced in \cite{EG}.
For $\Psi(t) := t^{\alpha}\ (0 < \alpha < n)$, we write $S_{p,\alpha}$ instead
of $S_{p,\Psi}$ and $\eta_{p,\alpha}$ instead of $\eta_{p,\Psi}$. Observe that
$S_{1,\alpha} := S_{\alpha}$ --- the Stummel class introduced in \cite{DH, RZ}.

The following two propositions confirm that $\eta_{p,\Psi}f$ is continuous (hence
measurable) and satisfies the doubling condition.

\begin{prop}\label{pr 2.1}
If $f \in S_{p,\Psi}$, then $\eta_{p,\Psi}f$ is continuous on $(0,\infty)$.
\end{prop}

\begin{proof}
Let $\{ r_{k} \}$ be a sequence in $(0,\infty)$ with $r_{k} \to r \in (0,\infty)$
and $x \in \R^{n}$. Choose $ r_{*} > 0 $ such that $r, r_{k} \leq r_{*}$
for every $ k \in \N$. Next, for every $ k \in \N $, define
\begin{equation*}
g_{k}(y) := \frac{|f(y)|^{p} \Psi(|y-x|)}{|y-x|^{n}} \chi_{B(x,r_{k})}(y) \quad{\rm and}\quad
g(y) := \frac{|f(y)|^{p} \Psi(|y-x|)}{|y-x|^{n}} \chi_{B(x,r)}(y),
\end{equation*}
for $ y \in B(x,r_{*}) $. We see that $ \{ g_{k} \} $ is a sequence of
nonnegative measurable functions on $ B(x,r_{*}) $, and $ g_{k} \to g $
almost everywhere on $B(x,r_{*})$. By the Dominated Convergence Theorem we obtain
\begin{equation*}
\int_{|y-x|<r_{*}} g_{k}(y) dy \rightarrow \int_{|y-x|<r_{*}} g(y) dy.
\end{equation*}
Therefore
\begin{equation}\label{eq 2.1}
\left(\int_{|y-x|<r_{k}} \frac{|f(y)|^{p} \Psi(|y-x|)}{|y-x|^{n}} dy
\right)^{\frac{1}{p}}  \rightarrow \left( \int_{|y-x|<r} \frac{|f(y)|^{p}
\Psi(|y-x|)}{|y-x|^{n}} dy \right)^{\frac{1}{p}}.
\end{equation}
Let $\epsilon$ be any positive real number. By (\ref{eq 2.1}), there exists $k_{0}
\in \N$ such that for all $ k \in \N $ with $k \geq k_{0}$ we have
\begin{align*}
\left(\int_{|y-x|<r} \frac{|f(y)|^{p} \Psi(|y-x|)}{|y-x|^{n}} dy\right)^{\frac{1}{p}} - \epsilon
& < \left(\int_{|y-x|<r_{k}} \frac{|f(y)|^{p} \Psi(|y-x|)}{|y-x|^{n}} dy\right)^{\frac{1}{p}} \\
& < \left(\int_{|y-x|<r} \frac{|f(y)|^{p} \Psi(|y-x|)}{|y-x|^{n}} dy\right)^{\frac{1}{p}} + \epsilon.
\end{align*}
Since $x \in \R^{n}$ is arbitrary, we conclude that
\begin{equation*}
\eta_{p,\Psi}f(r) - \epsilon \leq \eta_{p,\Psi}f(r_{k}) \leq \eta_{p,\Psi}f(r) + \epsilon.
\end{equation*}
Thus, we have proved that $\eta_{p,\Psi}f(r_{k}) \rightarrow \eta_{p,\Psi}f(r)$ for any
sequence $\{ r_{k} \}$ in $(0,\infty)$ with $ r_{k} \rightarrow r \in (0,\infty) $.
This means that $\eta_{p,\Psi}f$ is continuous on $(0,\infty)$.
\end{proof}

\begin{prop}\label{pr 2.2}
Let $\Psi$ satisfy the condition (\ref{eq 1.3}). If $f \in S_{p,\Psi}$, then $\eta_{p,\Psi}f$
satisfies the doubling condition.
\end{prop}

\begin{proof}
Let $x \in \R^{n}$ and $r > 0$. Choose $m = m(n) \in \N$ and $x_{1}, \dots, x_{m}\in B(x,r)$
such that
\begin{equation*}
B(x,r) \subseteq \bigcup_{i = 1}^{m} B\left( x_{i}, \frac{r}{2} \right).
\end{equation*}
Note that
\begin{align}\label{eq 2.2}
\left( \int_{|y-x|<r} \frac{|f(y)|^{p} \Psi(|y-x|)}{|y-x|^{n}} dy \right)^{\frac{1}{p}}
& \leq \sum_{i=1}^{m} \left( \int_{|y-x_{i}|<\frac{r}{2}}  \frac{|f(y)|^{p}
\Psi(|y-x|)}{|y-x|^{n}} dy\right)^{\frac{1}{p}} \\
& = \sum_{i=1}^{m} I_{i}. \nonumber
\end{align}
For $ i = 1, \dots,m $, we have
\begin{align}\label{eq 2.3}
I_{i}
& = \left( \int_{|y-x_{i}|<\frac{r}{2}} \frac{|f(y)|^{p} \Psi(|y-x|)}{|y-x|^{n}} dy \right)^{\frac{1}{p}} \\
& \leq \left( \int_{|y-x| > |y-x_{i}|, |y-x_{i}|<\frac{r}{2}} \frac{|f(y)|^{p} \Psi(|y-x|)}{|y-x|^{n}} dy
\right)^{\frac{1}{p}} \nonumber \\
& \qquad +  \left( \int_{|y-x| \leq |y-x_{i}|<\frac{r}{2}} \frac{|f(y)|^{p} \Psi(|y-x|)}{|y-x|^{n}} dy
\right)^{\frac{1}{p}} \nonumber \\
& = A_{i} + B_{i}. \nonumber
\end{align}
By the condition (\ref{eq 1.3}) on $ \Psi $, we obtain
\begin{align*}
A_{i} & = \left( \int_{|y-x| > |y-x_{i}|, |y-x_{i}|<\frac{r}{2}} \frac{|f(y)|^{p}
\Psi(|y-x|)}{|y-x|^{n}} dy\right)^{\frac{1}{p}} \\
& \leq c(p) \left( \int_{|y-x| > |y-x_{i}|, |y-x_{i}|<\frac{r}{2}} \frac{|f(y)|^{p}
\Psi(|y-x_{i}|)}{|y-x_{i}|^{n}} dy\right)^{\frac{1}{p}} \\
& \leq c(p) \left( \int_{|y-x_{i}|<\frac{r}{2}} \frac{|f(y)|^{p}
\Psi(|y-x_{i}|)}{|y-x_{i}|^{n}} dy\right)^{\frac{1}{p}}
\leq c(p) \, \eta_{p,\Psi}f\left( \frac{r}{2} \right).
\end{align*}
It is clear that
\begin{align*}
B_{i} \leq \left( \int_{|y-x|<\frac{r}{2}} \frac{|f(y)|^{p}
\Psi(|y-x|)}{|y-x|^{n}} dy\right)^{\frac{1}{p}}
\leq \eta_{p,\Psi}f\left( \frac{r}{2} \right).
\end{align*}
From (\ref{eq 2.2}) and (\ref{eq 2.3}), we get
\begin{align}\label{eq 2.4}
\left( \int_{|y-x|<r} \frac{|f(y)|^{p} \Psi(|y-x|)}{|y-x|^{n}} dy \right)^{\frac{1}{p}}
& \leq m(n)(c(p)+1) \, \eta_{p,\Psi}f\left( \frac{r}{2} \right) \\
& = c(n,p) \, \eta_{p,\Psi}f\left( \frac{r}{2} \right). \nonumber
\end{align}
Since the inequality (\ref{eq 2.4}) holds for all $ x \in \R^{n} $, we obtain
\begin{equation*}
\eta_{p,\Psi}f(r) \leq c(n,p) \, \eta_{p,\Psi}f\left( \frac{r}{2} \right).
\end{equation*}
According to the fact that $ \eta_{p,\Psi}f $ is nondecreasing, we conclude that
$ \eta_{p,\Psi}f $ satisfies the doubling condition.
\end{proof}

\section{Inclusion Between Generalized Stummel Classes}

In this section, we are going to investigate the inclusion between two Stummel classes.
The first proposition discusses the relationship between Stummel classes with different
parameters $ \Psi $. (Unless otherwise stated, we always assume that $1\le p<\infty$.)

\begin{prop}\label{pr 3.1}
Suppose that $ \Psi_{2} $ satisfies the condition (\ref{eq 1.3}) and that there exist $c>0$
and $ \delta > 0 $ such that $ \Psi_{2}(t) \leq c\,\Psi_{1}(t) $ for every $t \in (0,\delta)$.
Then $ S_{p,\Psi_{1}} \subseteq S_{p,\Psi_{2}}$.
\end{prop}

\begin{proof}
Let $ f \in S_{p,\Psi_{1}} $, $ x \in \R^{n} $, and $ r > 0 $. For $ r \le \delta $, we have
\[
\left(  \int_{|y-x|<r} \frac{|f(y)|^{p} \Psi_{2}(|y-x|)}{|y-x|^{n}} dy \right)^{\frac{1}{p}}
\leq c^{\frac{1}{p}} \left(  \int_{|y-x|<r} \frac{|f(y)|^{p} \Psi_{1}(|y-x|)}{|y-x|^{n}} dy
\right)^{\frac{1}{p}},
\]
whence $ \eta_{p,\Psi_{2}}f(r) \leq c^{\frac{1}{p}} \, \eta_{p,\Psi_{1}}f(r) \searrow 0 $ for
$ r \searrow 0 $. Hence $ f \in S_{p,\Psi_{2}} $.
\end{proof}

As an immediate consequence of Proposition \ref{pr 3.1}, we have the following corollary.

\begin{cor}\label{cor 3.2}
If $0<\alpha \le \beta<n$, then $S_{p,\alpha} \subseteq S_{p,\beta}$.
\end{cor}

\begin{remark}
For $0<\alpha<\beta<n$, the above inclusion is proper. Indeed, for $0<\beta<n$, define
$ f : \R^{n} \longrightarrow \R $ by the formula
\begin{equation*}
f(y) := \left( \frac{\chi_{B}(y)}{ |y|^{\beta} \, |\ln|y||^{2} } \right)^{\frac{1}{p}},
\quad y \in \R^{n},
\end{equation*}
where $B := B(0,e^{-\frac{2}{\beta}})$. Then $ f \in S_{p,\beta} \backslash S_{p,\alpha}$
whenever $0<\alpha<\beta$.
\end{remark}

The next proposition shows the relationship between two Stummel classes with different
parameters $p$.

\begin{prop}\label{pr 3.2}
If $ 1 \leq p_2 \le p_1 <\infty$ and $ \Psi $ satisfies (\ref{eq 1.1}),
then $ S_{p_1,\Psi} \subseteq S_{p_2,\Psi}$.
\end{prop}

\begin{proof}
Let $ f \in S_{p_1,\Psi}$, $ x \in \R^{n} $, and $ 0 < r \leq 1 $. Then by H\"{o}lder's inequality
we have
\begin{align*}
\int_{|y-x|<r} \frac{|f(y)|^{p_2} \Psi(|y-x|)}{|y-x|^{n}} \, dy
& \leq \left(\int_{|y-x|<r} \frac{|f(y)|^{p_1} \Psi(|y-x|)}{|y-x|^{n}}\, dy\right)^{\frac{p_2}{p_1}} \\
& \qquad \times \left(\int_{|y-x|<r} \frac{\Psi(|y-x|)}{|y-x|^{n}}\, dy\right)^{1-\frac{p_2}{p_1}} \\
& = \left(\int_{|y-x|<r} \frac{|f(y)|^{p_1} \Psi(|y-x|)}{|y-x|^{n}}\, dy\right)^{\frac{p_2}{p_1}} \\
& \qquad \times \left(c(n) \int_{0}^{r} \frac{\Psi(t)}{t}\, dt \right)^{1-\frac{p_2}{p_1}}.
\end{align*}
Therefore
\begin{align*}
\eta_{p_2,\Psi}f(r) \leq c(n,p_1,p_2) \, \eta_{p_1,\Psi}f(r) \, \left(\int_{0}^{r} \frac{\Psi(t)}{t}\, dt
\right)^{\frac{1}{p_2} - \frac{1}{p_1}} \searrow 0 \quad \text{for} \quad r \searrow 0,
\end{align*}
which tells us that $ f \in S_{p_2,\Psi}$. We conclude that $ S_{p_1,\Psi} \subseteq S_{p_2,\Psi}$.
\end{proof}

As a consequence of Proposition \ref{pr 3.2}, we have the following corollary.

\begin{cor}\label{cor 3.3}
If $1\leq p_2 \le p_1< \infty $, then $S_{p_1,\alpha}\subseteq S_{p_2,\alpha}$.
\end{cor}

\begin{remark}
For $1\le p_2<p_1<\infty$, the above inclusion is proper. Indeed, for
$\frac{\alpha}{p_1}< \gamma <\min \{ \frac{\alpha}{p_2}, \frac{n}{p_1} \}$, we have
$f(y):=|y|^{-\gamma}\in S_{p_2,\alpha} \backslash S_{p_1,\alpha}$.
\end{remark}

From Proposition \ref{pr 3.1} and Proposition \ref{pr 3.2}, we get the following corollary.

\begin{cor}
Suppose that $ 1 \leq p_2 \le p_1<\infty $, $ \Psi_2$ satisfies the conditions
(\ref{eq 1.1}) and (\ref{eq 1.3}),
and there exist $c>0$ and $\delta > 0$ such that $ \Psi_2(t) \leq c \, \Psi_1(t) $ for every
$t \in (0,\delta)$. Then $ S_{p_1,\Psi_1} \subseteq S_{p_2,\Psi_2}$.
\end{cor}

\section{Inclusion Between Stummel Classes and Morrey Spaces}

Our next theorem gives an inclusion relation between generalized Morrey spaces and
generalized Stummel classes. We also give an example of a function that belongs to
the generalized Stummel class but not to the generalized Morrey space.

\begin{theorem}\label{th:23518-1}
Let $1\le p_2\le p_1<\infty$. Assume that $\Psi_1$ satisfies \eqref{eq 1.2} and that
$\Psi_2$ satisfies the right-doubling condition \eqref{eq 1.4}. If
\begin{align}\label{eq:23518-0}
\int_0^1 \frac{\Psi_1(t)^{p_2}\Psi_2(t)}{t} \ dt<\infty,
\end{align}
then  $ L^{p_1,\Psi_1} \subseteq S_{p_2,\Psi_2}$.
\end{theorem}

\begin{remark}
Let $p_1=p_2=1$, $\Psi_1(t):=t^{\lambda-n}$ where $0\le\lambda\le n$, and $\Psi_2(t):=t^\alpha$
where $n-\lambda<\alpha<n$. Then the above theorem reduces to the result in \cite[p.~56]{RZ}.
\end{remark}

\begin{proof}[Proof of Theorem \ref{th:23518-1}]
Let $f\in L^{p_1,\Psi_1}$, $x \in \R^n$, and $r>0$. Since $\Psi_2$ satisfies \eqref{eq 1.4},
we have
\begin{align*}
\int_{|x-y|<r}\frac{|f(y)|^{p_2} \Psi_2(|x-y|)}{|x-y|^n}\ dy
&=
\sum_{k=-\infty}^{-1} \int_{2^kr\le |x-y|< 2^{k+1}r}
\frac{|f(y)|^{p_2} \Psi_2(|x-y|)}{|x-y|^n}\ dy
\\
&\le
c \sum_{k=-\infty}^{-1}\frac{ \Psi_2(2^kr)}{|B(x,2^{k+1}r)|}
\int_{B(x,2^{k+1}r)}|f(y)|^{p_2}\ dy.
\end{align*}
Combining the last inequality and H\"older's inequality, we get
\begin{align}\label{eq:23518-1}
\int_{|x-y|<r}\frac{|f(y)|^{p_2} \Psi_2(|x-y|)}{|x-y|^n}\ dy
&\le
c \sum_{k=-\infty}^{-1}\frac{\Psi_2(2^kr)}{|B(x,2^{k+1}r)|^{p_2/p_1}}
\|f\|_{L^{p_1}(B(x,2^{k+1}r))}^{p_2}
\nonumber
\\
&\le
c\,\|f\|_{L^{p_1,\Psi_1}}^{p_2}
\sum_{k=-\infty}^{-1} \Psi_1(2^{k+1}r)^{p_2}\Psi_2(2^kr).
\end{align}
Using \eqref{eq 1.4} and the monotonicity of $\Psi_1$, we get
\begin{align}\label{eq:23518-2}
\sum_{k=-\infty}^{-1} \Psi_1(2^{k+1}r)^{p_2}\Psi_2(2^kr)
\le
c \sum_{k=-\infty}^{-1} \int_{2^{k-1}r}^{2^kr} \frac{ \Psi_1(t)^{p_2}\Psi_2(t)}{t}\ dt
=
c \int_0^{r/2} \frac{ \Psi_1(t)^{p_2}\Psi_2(t)}{t}\ dt.
\end{align}
We combine \eqref{eq:23518-1} and \eqref{eq:23518-2} to obtain
\begin{align}\label{eq:23518-3}
\eta_{p_2,\Psi_2}f(r)
\le
c \left(\int_0^{r/2} \frac{ \Psi_1(t)^{p_2}\Psi_2(t)}{t}\ dt\right)^{\frac{1}{p_2}}
\|f\|_{L^{p_1, \Psi_1}}.
\end{align}
Since $\displaystyle \int_0^{1} \frac{ \Psi_1(t)^{p_2}\Psi_2(t)}{t}\ dt<\infty$, we see that
$\displaystyle \lim_{r\to 0^+}\int_0^{r/2} \frac{ \Psi_1(t)^{p_2}\Psi_2(t)}{t}\ dt=0$.
This fact and \eqref{eq:23518-3} imply $\lim\limits_{r\to 0^+} \eta_{p_2,\Psi_2}f(r) =0$.
Hence, $f\in S_{p_2,\Psi_2}$. This shows that $L^{p_1, \Psi_1}  \subseteq S_{p_2,\Psi_2}$.
\end{proof}

The following example shows that the inclusion in Theorem \ref{th:23518-1} is proper.

\begin{example}\label{ex: 1}
Let $ 1 \leq p_{2} \leq p_{1} < \infty $, $ \Psi_{2} $ satisfy the condition (\ref{eq 1.4}),
$ \Psi_{2}(t) \, |\ln(t)|^{2} $ be nondecreasing on $(0,\delta)$ for some $ \delta > 0 $,
and $ \Psi_{1}(r)^{p_2} \, \Psi_{2}(r) \, |\ln(r)|^{2} \searrow 0 $ as $ r \searrow 0 $.
Define $ f : \R^{n} \longrightarrow \R $ by the formula
\begin{equation*}
f(y) := \left( \frac{\chi_{B}(y)}{ \Psi_{2}(|y|) \, |\ln|y||^{2} } \right)^{\frac{1}{p_2}},
\quad y \in \R^{n},
\end{equation*}
where $B := B(0,\delta)$. Then $ f \in S_{p_2,\Psi_{2}} \backslash L^{p_1,\Psi_{1}} $.

First we show that $f\in S_{p_2,\Psi_{2}}$. Let $0<r<\min \{\delta, 1\}$.
Since the function $f$ is radial and nonincreasing, the supremum in the Stummel modulus
is attained at the origin, so that
\[
\eta_{p_2,\Psi_{2}}f(r)
=
\left(
\int_{|y|<r}
\frac{|f(y)|^{p_2} \Psi_2(|y|)}{|y|^n} \ dy
\right)^{\frac{1}{p_2}}
=
\left(
\int_{|y|<r}
\frac{1}{|\ln |y||^2|y|^n} \ dy
\right)^{\frac{1}{p_2}}.
\]
Converting to polar coordinates, we get
\[
\int_{|y|<r}
\frac{1}{|\ln |y||^2|y|^n} \ dy
=c\int_0^r \frac{1}{s (\ln s)^2} \ ds
=-\frac{c}{\ln r}.
\]
Therefore,
\[
\eta_{p_2,\Psi_{2}}f(r)
=
c \left(-\frac{1}{\ln r}\right)^{\frac{1}{p_2}}.
\]
Since $ \lim\limits_{r\to0^+} \frac{1}{\ln r} = 0 $, we conclude that
$ \eta_{p_2,\Psi_{2}}f(r) \searrow 0 $ for $ r \searrow 0 $. This proves
that $f \in S_{p_2,\Psi_{2}}$.

Now, we will show that $ f \notin L^{p_1,\Psi_{1}} $. Let $ 0 < r < \delta $.
Since $ \Psi_{2}(t) \, |\ln(t)|^{2} $ is nondecreasing on $ (0,r) \subseteq (0,\delta) $,
we have
\begin{align*}
\frac{1}{\Psi_{1}(r)^{p_1}} \frac{1}{ |B(0,r)| } \int_{B(0,r)} |f(y)|^{p_1} \, dy
& = \frac{1}{\Psi_{1}(r)^{p_1}} \frac{1}{ |B(0,r)| } \int_{B(0,r)} \left( \frac{1}
{ \Psi_{2}(|y|) \, |\ln|y||^{2} } \right)^{\frac{p_1}{p_2}}  \, dy \\
& \geq \frac{1}{\Psi_{1}(r)^{p_1} \, |B(0,r)|} \left( \frac{1}{\Psi_{2}(r) \,
|\ln r|^{2}}\right)^{\frac{p_1}{p_2}}  \int_{B(0,r)} \, dy \\
& = \left( \frac{1}{\Psi_{1}(r)^{p_2} \, \Psi_{2}(r) \, |\ln r|^{2} }
\right)^{\frac{p_1}{p_2}}.
\end{align*}
Note that $ \Psi_{1}(r)^{p_2} \, \Psi_{2}(r) \, |\ln r|^{2} \searrow 0 $  as
$ r \searrow 0 $. Then
\begin{equation*}
\left( \frac{1}{\Psi_{1}(r)^{p_2} \, \Psi_{2}(r) \, |\ln r|^{2} }
\right)^{\frac{p_1}{p_2}} \rightarrow \infty \quad \text{for} \quad r \searrow 0.
\end{equation*}
We conclude that $ f \notin L^{p_1,\Psi_{1}} $. \qed
\end{example}

\begin{remark}
Let $ 1 \leq p_{2} \leq p_{1} < \infty $, $ \Psi_{1}(t) := t^{\frac{\lambda-n}{p_{1}}} $ where
$0 \le \lambda \le n $, and $ \Psi_{2}(t) := t^\alpha $ where $\frac{(n-\lambda)p_2}{p_1}
< \alpha < n $. It can be shown that $ \Psi_{1} $ and $ \Psi_{2} $ satisfy all conditions
in Theorem \ref{th:23518-1} and Example \ref{ex: 1}.
\end{remark}

As a counterpart of Theorem \ref{th:23518-1}, we have the following result.

\begin{theorem}\label{th:28818-1}
Let $1\le p_2\le p_1<\infty$ and assume that $\Psi_1$ satisfies
\eqref{eq 1.3}. If $f\in S_{p_1,\Psi_1}$ and
\begin{align}\label{eq:28818-1}
\eta_{p_1,\Psi_{1}}f(r)\le c \Psi_{1}(r)^{\frac{1}{p_1}} \Psi_{2}(r)
\end{align}
for some $\Psi_{2}$ and for every $r>0$, then $f\in L^{p_2, \Psi_2}$.
\end{theorem}

\begin{proof}
Let $a\in \mathbb{R}^n$ and $r>0$. Then, by H\"older's inequality, we have
\begin{align*}
\int_{B(a,r)} |f(x)|^{p_2} \ dx
&\le
c\,r^{n\left( 1-\frac{p_2}{p_1}\right)}
\left(
\int_{B(a,r)} |f(x)|^{p_1} \ dx
\right)^{\frac{p_2}{p_1}}
\\
&=
\frac{c\,r^n}{\Psi_{1}(r)^{\frac{p_2}{p_1}}}
\left(
\int_{B(a,r)} \frac{|f(x)|^{p_1} \Psi_1(r)}{r^n} \ dx
\right)^{\frac{p_2}{p_1}}.
\end{align*}
We combine \eqref{eq 1.3}, \eqref{eq:28818-1}, and Definition \ref{D1} to obtain
\begin{align*}
\int_{B(a,r)} |f(x)|^{p_2} \ dx
&\le
\frac{c\,r^n}{\Psi_{1}(r)^{\frac{p_2}{p_1}}}
\left(
\int_{B(a,r)} \frac{|f(x)|^{p_1} \Psi_1(|x-a|)}{|x-a|^n} \ dx
\right)^{\frac{p_2}{p_1}}
\\
&\le
\frac{c\,r^n}{\Psi_{1}(r)^{\frac{p_2}{p_1}}}
[\eta_{p_1,\Psi_{1}}f(r)]^{p_2}
\le
c\,r^{n}\Psi_2(r)^{p_2}.
\end{align*}
Consequently,
\begin{align*}
\frac{1}{|B(a,r)|\Psi_{2}(r)}
\left(
\int_{B(a,r)} |f(x)|^{p_2} \ dx
\right)^{\frac{1}{p_2}}
\le c.
\end{align*}
Since $a$ and $r$ are arbitrary, we conclude that $f\in L^{p_2, \Psi_{2}}$.
\end{proof}

Taking $\Psi_1(t):=t^{\alpha}$ and $\Psi_{2}(t):=t^{\frac{\sigma}{p_2}-\frac{\alpha}{p_1}}$
where $0<\alpha<n$, $1\le p_2\le p_1<\infty$, and $0<\sigma <\frac{\alpha p_2}{p_1}$, we get
the following corollary.

\begin{cor}\label{cor:280818-1}
Let $1\le p_2\le p_1<\infty$ and $0<\alpha<n$. If $f\in S_{p_1,\alpha}$ and $\eta_{p_1,\alpha}f(r)
\le cr^{\frac{\sigma}{p_2}}$ for some $0<\sigma<\frac{\alpha p_2}{p_1}$ and for every $r>0$, then
$f\in L^{p_2, n+\sigma-\frac{\alpha p_2}{p_1}}$.
\end{cor}

Next, we are going to investigate the relation between generalized Stummel classes and
generalized weak Morrey spaces. The generalized weak Morrey spaces are defined as follows.

\begin{defi}
Let $1\le p<\infty$ and $\Psi:(0,\infty) \to (0,\infty)$.
The \textbf{generalized weak Morrey space} $wL^{p, \Psi}=wL^{p, \Psi}(\R^n)$
is defined to be the set of all measurable functions $f$ for which
\[
\|f\|_{wL^{p,\Psi}}
:=
\sup_{a\in \R^n,\,r>0,\, t>0}
\frac{t|\{x\in B(a,r): |f(x)|>t\}|^{1/p}}{\Psi(r)|B(a,r)|^{1/p}}<\infty
\]
\end{defi}

The inclusion between generalized Stummel Classes and generalized weak Morrey spaces
is given in the following theorems.

\begin{theorem}\label{th-160618}
Let $1\le p_2<p_1<\infty$. Assume that $\Psi_1$ satisfies \eqref{eq 1.2} and that
$\Psi_2$ satisfies \eqref{eq 1.4}. If
\[
\int_{0}^1 \frac{\Psi_1(t)^{p_2}\Psi_2(t) }{t} \ dt<\infty,
\]
then $wL^{p_1, \Psi_1} \subseteq S_{p_2,\Psi_2}$.
\end{theorem}

\begin{proof}
Since $p_2<p_1$, by virtue of \cite[Theorem 5.1]{GHLM}, we have
$wL^{p_1,\Psi_1}\subseteq L^{p_2, \Psi_1}$.
By Theorem \ref{th:23518-1}, we have $L^{p_2, \Psi_1}
\subseteq S_{p_2,\Psi_2}$. It thus follows that $wL^{p_1,\Psi_1}
\subseteq S_{p_2,\Psi_2}$.
\end{proof}

\begin{theorem}\label{th:28818-2}
Let $1\le p_1\le p_2<\infty$ and assume that $\Psi_1$ satisfies
\eqref{eq 1.3}. If $f\in S_{p_1,\Psi_1}$ and the inequality
\eqref{eq:28818-1} holds
for some $\Psi_{2}$ and for every $r>0$, then $f\in wL^{p_2, \Psi_2}$.
\end{theorem}

\begin{proof}
The assertion follows from Theorem \ref{th:28818-1} and the inclusion
$L^{p_2, \Psi_2}\subseteq wL^{p_2, \Psi_2}$.
\end{proof}

For the classical weak Morrey spaces and Stummel classes, we have the following result.

\begin{theorem}\label{th_4.2}
For $1\le p_2<p_1<\infty$, if $0\le \lambda<n$ and $\frac{(n-\lambda)p_2}{p_1}<\alpha<n$,
then $wL^{p_1,\lambda} \subseteq S_{p_2, \alpha}$.
Conversely, for $1 \leq p < \infty $, if $f \in S_{p,\alpha}$ for $0<\alpha<n$ and
$\eta_{p,\alpha}f(r) \leq c r^{\frac{\sigma}{p}}$ for some $\sigma > 0$, then
$f \in wL^{p,n-\alpha+\sigma}$.
\end{theorem}

\begin{proof}
The first assertion follows from Theorem \ref{th-160618} by taking
$\Psi_1(t):=t^{\frac{\lambda-n}{p_1}}$, and $\Psi_2(t):=t^\alpha$ where
$0\le\lambda< n$ and $\frac{(n-\lambda)p_2}{p_1}<\alpha<n$.
The second part is a consequence of Corollary \ref{cor:280818-1} when $p_1=p_2=p$
and the inclusion $L^{p,n-\alpha+\sigma}\subseteq wL^{p,n-\alpha+\sigma}$.
\end{proof}

\begin{remark}
The second part of Theorem \ref{th_4.2} generalizes the result in \cite[p.~57]{RZ}.
For the case $ p = 1 $, the first part of Theorem \ref{th_4.2} does not generally hold.
To see this, consider the function $f(y) := |y|^{-n},\ y \in \R^{n}$.
Then $ f \in  wL^{1,\lambda} $ for $0\le\lambda<n$, but $ f \notin S_{\alpha} $
for $n-\lambda<\alpha<n$.
\end{remark}

\section{Further Results}

In this section, we study the relation between bounded Stummel modulus classes
$\tilde{S}_{p,\alpha}$ and Stummel classes. We also study the inclusion between
$\tilde{S}_{p,\alpha}$ and Lorentz spaces. For $0< \alpha <n$ and $1 \leq p < \infty$,
recall the definition of the Stummel modulus
\begin{equation*}
\eta_{p,\alpha}f(r):=\sup_{x \in \R^{n}} \left(
\int_{|x-y|<r} \frac{|f(y)|^p}{|x-y|^{n-\alpha}} \,dy
\right)^{\frac1p},\quad r>0.
\end{equation*}

\begin{defi}\label{Stumod}
For $0< \alpha <n$ and $1 \leq p < \infty$, we define \textbf{the bounded Stummel
modulus class} $\tilde{S}_{p,\alpha}=\tilde{S}_{p,\alpha}(\R^n)$ by
\begin{align*}
\tilde{S}_{p,\alpha} :=  \left\{ f\in L_{\rm loc}^{p}(\R^n): \eta_{p,\alpha}f(r) < \infty
\quad \text{for all} \quad r > 0 \right\}.
\end{align*}
\end{defi}

Note that the inclusions similar to Corollary \ref{cor 3.2} and Corollary \ref{cor 3.3}
also hold for $\tilde{S}_{p,\alpha}$. Moreover, we have $S_{p,\alpha} \subseteq
\tilde{S}_{p,\alpha}$. This inclusion is proper due to the following example which we
adapt from \cite[p. 250--251]{AS}.

\begin{example}
Let $ 0 < \alpha < n $ and $ 1 \leq p < \infty $. For every $ k \in \N $ with $ k \geq 3 $,
let $x_{k} :=(2^{-k}, 0, \dots,0 ) \in \R^{n} $ and
\begin{equation*}
V_{k}(y) := \begin{cases}
8^{ \alpha k}, & y \in B(x_{k},8^{-k}), \\
0, & y \notin B(x_{k},8^{-k}).
\end{cases}
\end{equation*}
Define $ V: \R^{n} \rightarrow \R $ by the formula
\begin{equation*}
V(y) := \left( \sum_{k=3}^{\infty} V_{k}(y) \right)^{\frac{1}{p}}.
\end{equation*}
Since
\begin{align*}
\int_{B(x,r)} |V(y)|^{p} \, dy
= \sum_{k=3}^{\infty } \int_{B(x,r)} |V_{k}(y)| \, dy
\leq c(n) \sum_{k=3}^{\infty } 8^{( \alpha -n)k} < \infty
\end{align*}
for every $ x \in \R^{n} $ and $ r>0 $ where $ c(n) := |B(0,1)| $, we obtain
$ V \in  L_{\rm loc}^{p}(\R^n) $.

We will show that $ V \in \tilde{S}_{p,\alpha}$. Let
\begin{equation*}
\rho_{k}(x) := \int_{\R^{n}} \frac{|V_{k}(y)|}{|x-y|^{n-\alpha}} \,dy
= 8^{\alpha k} \, \int_{|y-x_{k}|<8^{-k}} \frac{1}{|x-y|^{n-\alpha}} \,dy, \quad x \in \R^{n}.
\end{equation*}
There are two cases: (i) $ |x-x_{k}| \geq 2^{-2k+1} $, or, (ii) $ |x-x_{k}| < 2^{-2k+1}$.

Suppose that the case (i) holds, that is, $|x-x_{k}| \geq 2^{-2k+1}$. We have,
\begin{align}\label{eq 5.1}
\rho_{k}(x)
\leq c(n) 2^{(\alpha-n)k}.
\end{align}
For the case (ii) $ |x-x_{k}| < 2^{-2k+1}$, we have
\begin{align}\label{eq 5.2}
\rho_{k}(x) \leq c(n,\alpha)
\end{align}
where $ c(n,\alpha) := \max \{ c(n),\frac{3^{\alpha}}{\alpha} c(n)\} $.

Given $ x \in \R^{n} $, we have $ x \notin B(x_{k},2^{-2k+1}) $ for all $ k \geq 3 $,
or $ x \in B(x_{j},2^{-2j+1}) $ for some $ j \geq 3 $. Assume that $ x \notin B(x_{k},2^{-2k+1}) $
for all $ k \geq 3 $. Hence, from \eqref{eq 5.1}, we have
\begin{align}\label{eq 5.3}
\int_{\R^{n}} \frac{|V(y)|^{p}}{|x-y|^{n-\alpha}} \, dy \leq \sum_{k=3}^{\infty} \rho_{k}(x)
\leq c(n) \sum_{k=3}^{\infty} 2^{(\alpha-n)k} < \infty.
\end{align}
Now assume that $x \in B(x_{j},2^{-2j+1})$ for some $j \ge 3$. Since $\left\lbrace B(x_{k},2^{-2k+1})
\right\rbrace_{k \ge 3}$ is a disjoint collection, we find that there is only one $j \in \N$, $j \ge 3$,
such that $ x \in B(x_{j},2^{-2j+1}) $. Using \eqref{eq 5.1} and \eqref{eq 5.2}, we get
\begin{align}\label{eq 5.4}
\int_{\R^{n}} \frac{|V(y)|^{p}}{|x-y|^{n-\alpha}} \, dy
& \leq c(n,\alpha) + \sum_{\substack{k=3 \\ k \neq j}}^{\infty} \rho_{k}(x) \\
& \leq c(n,\alpha) + c(n) \sum_{\substack{k=3 \\ k \neq j}}^{\infty} 2^{(\alpha-n)k} < \infty. \nonumber
\end{align}
According to \eqref{eq 5.3} and \eqref{eq 5.4}, for every $ r > 0 $, we have
\begin{equation*}
\int_{|x-y|<r} \frac{|V(y)|^{p}}{|x-y|^{n-\alpha}} \, dy
\leq \int_{\R^{n}} \frac{|V(y)|^{p}}{|x-y|^{n-\alpha}} \, dy < \infty.
\end{equation*}
Therefore $ \eta_{p,\alpha}V(r) < \infty $, and we conclude that $ V \in \tilde{S}_{p,\alpha}$.

Now, we will show that $ V \notin S_{p,\alpha}$. Let $r > 0$. By Archimedean property,
there is $ k \geq 3 $ such that $ 8^{-k} < r $. Note that
\begin{align*}
\left( \eta_{p,\alpha}V(r)\right)^{p}
& \ge \int_{|y-x_{k}|<r} \frac{|V(y)|^{p}}{|y-x_{k}|^{n-\alpha}} \,dy \\
& \geq \int_{|y-x_{k}|<r} \frac{|V_{k}(y)|}{|y-x_{k}|^{n-\alpha}} \,dy \\
& \geq 8^{\alpha k} \, \int_{|y-x_{k}|<8^{-k}} \frac{1}{|y-x_{k}|^{n-\alpha}} \,dy
= \frac{c(n)}{\alpha}.
\end{align*}
This shows that $\eta_{p,\alpha}V$ stays away from zero. Thus $V \notin S_{p,\alpha}$. \qed
\end{example}

Given a measurable function $ f: \R^{n} \rightarrow \R $, consider the
\textit{distribution function} $ D_{f} $ of $ f $ which is given by
\begin{align*}
D_{f}(\sigma) := \left| \left\lbrace x \in \R^{n} : |f(x)| > \sigma \right\rbrace  \right|,
\quad \sigma > 0.
\end{align*}
The \textit{decreasing rearrangement} of $f$ is the function $f^*$ defined on $[0, \infty)$ by
\begin{align*}
f^*(t) := \inf \left\lbrace \sigma : D_{f}(\sigma) \leq t  \right\rbrace, \quad t \geq 0.
\end{align*}

\begin{defi}
Let $ 0 <\kappa ,p \leq \infty $. The \textbf{Lorentz space} $ L_{\kappa}^p=L_{\kappa}^p(\R^n)$
is the collection of all measurable functions $ f: \R^n \rightarrow \R $ satisfying
$ \| f \|_{L_{\kappa}^p} < \infty $, where
\begin{equation*}
\| f \|_{L_{\kappa}^p} :=
\begin{cases}
\left( \int_{0}^{\infty} \left( t^{\frac{1}{\kappa}} f^*(t) \right)^{p} \frac{dt}{t}
\right)^{\frac{1}{p}}, &\quad \text{if}\ \ p < \infty, \\
\sup\limits_{ t > 0} t^{\frac{1}{\kappa}}f^*(t), &\quad \text{if}\ \ p = \infty.
\end{cases}
\end{equation*}
\end{defi}
Note that $L^{\infty}_\kappa=wL^{\kappa}$ for $\kappa\ge1$.
The following lemma is a well-known inclusion relation between Lorentz spaces
(see \cite[p.~49]{Gra} or \cite[p.~305]{Kuf} for its proof).

\begin{lemma}\label{lem 5.4}
If $ 0 < \kappa \leq \infty $ and $ 0 < p_{2} \le p_{1}\leq \infty $, then
$ L_{\kappa}^{p_2} \subseteq L_{\kappa}^{p_1} $.
\end{lemma}

Moreover, we have the following relation between Lorentz spaces and
bounded Stummel modulus classes.

\begin{lemma}\label{le:Castillo}\cite[Lemma 2.7]{CRR}
Let $0<\alpha <n$. Then $L^{1}_{\frac{n}{\alpha}} \subseteq \tilde{S}_{1,\alpha}$. 	
\end{lemma}

Our theorem below is an extension of Lemma \ref{le:Castillo}.

\begin{theorem}\label{thm5.6}
Let $1\le p<\infty$ and $0<\alpha<n$. If $\frac{np}{\alpha}\le \kappa<\infty$,
then
\[
L^{p}_{\kappa}
\subseteq
\tilde{S}_{p,\alpha}.
\]
\end{theorem}

\begin{proof}
We first prove the case where $\kappa=\frac{np}{\alpha}$. Let $f \in L^{p}_{\frac{np}{\alpha}}$.
Then $|f|^{p}\in L^{1}_{\frac{n}{\alpha}}$. By virtue of Lemma \ref{le:Castillo}, we have
$|f|^{p} \in \tilde{S}_{1,\alpha}$. According to Definition \ref{Stumod}, we see that
$f\in \tilde{S}_{p,\alpha}$. Thus, we obtain $L^{p}_{\frac{np}{\alpha}} \subseteq \tilde{S}_{p,\alpha}$.

Let us now consider the case where $\kappa>\frac{np}{\alpha}$.
Since $0<\alpha<n$, we have $\kappa>p$.
Hence by Theorem \ref{th_4.2} (for $ \lambda = 0 $),
we obtain $wL^{\kappa}\subseteq S_{p, \alpha}$. Now, combining this with
Lemma \ref{lem 5.4} and the remark after Definition \ref{Stumod}, we see that
\[
L^{p}_{\kappa}\subseteq wL^{\kappa}\subseteq S_{p, \alpha}\subseteq
\tilde{S}_{p, \alpha}.
\]
This completes the proof.
\end{proof}

\begin{remark}
For $\frac{n}{\alpha}<\kappa<\infty$, we observe that $L^1_\kappa \not\subseteq L^1_{\frac{n}{\alpha}}$.
To see this, one may check that $f(x):=|x|^{-\alpha}\chi_{\{x\,:\,|x|>1\}} \in L^1_\kappa 
\setminus L^1_{\frac{n}{\alpha}}$.
\end{remark}

\begin{remark}
It follows from Theorem \ref{thm5.6} that, for $1\le p_2\le p_1<\infty$ and
$\frac{np_1}{\alpha}\le\kappa<\infty$, the inclusion
$L^{p_1}_{\kappa}\subseteq \tilde{S}_{p_2,\alpha}$ holds.
\end{remark}

\begin{remark}
By using the same trick as in the proof of the first part of Theorem  \ref{thm5.6}, one can extend \cite[Theorem 3.1]{G} to the corresponding function spaces with parameter $p$ instead of $1$.  
\end{remark}

\bigskip

\noindent{\bf Acknowledgement}. This research is supported by ITB Research \& Innovation
Program 2018.

\end{document}